\documentclass[12pt,a4paper]{article}
\usepackage[left=2cm, right=2cm, top=2.5cm, bottom=2.5cm]{geometry}
\usepackage{enumerate}
\usepackage{amssymb}
\usepackage{amsmath}
\usepackage{mathrsfs}
\usepackage{amsthm}
\usepackage{mathtools}
\usepackage{enumitem}
\theoremstyle{definition}

\newtheorem{theorem}{Theorem}[section]
\newtheorem{lemma}[theorem]{Lemma}
\newtheorem{proposition}[theorem]{Proposition}
\newtheorem{corollary}[theorem]{Corollary}

\usepackage[colorlinks,linkcolor=blue]{hyperref}
\numberwithin{equation}{section}

\usepackage{enumerate}

\allowdisplaybreaks[4]

\makeatletter

\newcommand{\Rmnum}[1]{\expandafter\@slowromancap\romannumeral #1@}
\makeatother

\begin{document}
	
	\title{The Michael-Simon-Sobolev inequality on manifolds
			for positive symmetric tensor fields}
	
	\author{
		Yuting Wu 
		\thanks{School of Mathematical Sciences, East China Normal University, 500 Dongchuan Road, Shanghai 200241,
			P. R. of China, E-mail address: 52215500001@stu.ecnu.edu.cn. }
		\and
		Chengyang Yi 
		\thanks{School of Mathematical Sciences, Tongji University, 
			1239 Siping Road, Shanghai 200092, P. R. of China, E-mail address: cyyi@tongji.edu.cn. }	
		\and
		Yu Zheng 
		\thanks{School of Mathematical Sciences, East China Normal University, 500 Dongchuan Road, Shanghai 200241,
			P. R. of China, E-mail address: zhyu@math.ecnu.edu.cn. }	
	}
	
	\date{}
	
	\maketitle
	
	\begin{abstract}
		We prove the Michael-Simon-Sobolev inequality for smooth symmetric uniformly positive definite $\left( 0,2\right)$-tensor fields on compact submanifolds with or without boundary
		in Riemannian manifolds with nonnegative sectional curvature
	    by the Alexandrov-Bakelman-Pucci (ABP) method.
		It should be a generalization of S. Brendle in \cite{Bren}.
	\end{abstract}

   	Keywords: the Michael-Simon-Sobolev inequality; 
   	compact submanifold; nonnegative sectional curvature.

	\section{Introduction}
	
	Geometric inequalities are indispensable tools in geometric analysis, serving as vital bridges to understanding, comparing, and quantifying geometric quantities.
	They empower mathematicians to prove theorems, derive properties, and solve complex problems.
    Two notable examples are the Li–Yau estimate and the Sobolev inequality.
	The Li–Yau estimate is a powerful tool for estimating the decay rate of solutions to the heat equation. 
	Recent advancements were made by Q. S. Zhang and X. L. Li, who introduced matrix Li–Yau–Hamilton estimates for positive solutions to the heat equation and the backward conjugate heat equation under Ricci flow (see \cite{QZ}), 
	their work builds upon earlier contributions (see \cite{CAO, CH, HA, PR, YU, QZ}).
	In parallel, 
    numerous experts have undertaken further research into the Sobolev inequality (see \cite{LL, Wang, Joh, Dong22, Ma, DP, Kr, BE}).
	In 2018, D. Serre (see \cite{DS}) established a new sharp Sobolev type inequality on a bounded convex domain $\Omega$ for a symmetric positive semi-definite (0,2)-tensor field in $\mathbb{R}^{n}$.
	In 2019, S. Brendle (see \cite{Bre}) proved a new Michael-Simon-Sobolev inequality for functions on compact submanifolds of arbitrary dimension and codimension in Euclidean space using the Alexandrov-Bakelman-Pucci (ABP) method, 
	the inequality is sharp when the codimension is at most 2, 
	resolving a long-standing conjecture.  
	Building upon this work, S. Brendle (see \cite{Bren}) extended his findings in \cite{Bre} to the scenario where the ambient space is a Riemannian manifold with nonnegative curvature. 
	Inspired by the contributions of D. Serre (see \cite{DS}) and S. Brendle (see \cite{Bre}),
	D. Pham (see \cite{DP}) established a Sobolev inequality involving a positive symmetric matrix-valued function $A$ on a smooth bounded convex domain in $\mathbb{R}^{n}$, albeit without the convexity condition of $\Omega$.

	In this paper, inspired by \cite{DP} and \cite{Bren},
	we present a new Michael-Simon type Sobolev inequality 
	which replaces the positive symmetric matrix-valued function in \cite{DP} with a smooth symmetric uniformly positive definite 
	$\left( 0,2\right)$-tensor field 
    on a compact submanifold in Riemannian manifolds with nonnegative sectional curvature
	using the ABP method.
	Our result represents a generalization of both D. Serre’s inequality and a extension of S. Brendle’s inequality.
	
	Let $M$ be a complete noncompact Riemannian manifold of dimension $k$ with nonnegative Ricci curvature. Denote by $\left| B^{k}\right| $ the volume of the unit ball in $\mathbb{R}^{k}$.
	The asymptotic volume ratio of $M$ is defined as
	\begin{equation*}  
		\theta:=\underset{r\rightarrow\infty} {\lim}\dfrac{\left\lbrace p\in M: d(p, q)\leq r\right\rbrace }{|B^{k}|r^{k}}
	\end{equation*}
	where $q$ is some fixed point on the manifold.
    According to the Bishop-Gromov volume comparison theorem,  
    the limit exists and that $\theta\leq1$.
    
    We get the following principal theorem:
	\begin{theorem}	\label{th:01}
		Let $M$ be a complete noncompact manifold of dimension $n+m$
		with nonnegative sectional curvature,
		where $m\geq2$.
		Let $\Sigma$ be a compact $n$-dimensional submanifold of $M^{n+m}$ with smooth boundary $\partial\Sigma$ (possibly  $\partial\Sigma=\varnothing$).
		If $A$ is a smooth symmetric uniformly positive definite $\left( 0,2\right)-$tensor field on $\Sigma$,  then
		\begin{equation}  \label{eq:1.1}
			\int_{\Sigma}\sqrt{\left| \mathrm{div}_{\Sigma}A\right|^{2}+\left| \left\langle A,\Rmnum{2}\right\rangle \right|^{2}}
			+\int_{\partial\Sigma}\left| A\left( \nu\right) \right| 
			\geq n\left[ \frac{\left( n+m\right)\left| B^{n+m}\right|}{m\left| B^{m}\right|}\right] ^{\frac{1}{n}}\theta^{\frac{1}{n}}\left( \int_{\Sigma}\left( \mathrm{det}A\right)^{\frac{1}{n-1}}\right) ^{\frac{n-1}{n}},
		\end{equation}
		where $\nu$ is the unit outer normal vector field on $\partial\Sigma$ with respect to $\Sigma$, 
	    $\Rmnum{2}$ is the second fundamental form of $\Sigma$,
	    and $\theta$ is the asymptotic volume ratio of $M$.
	\end{theorem}

    When $m=2$,
    because of $\left( n+2\right)\left| B^{n+2}\right|=2\left| B^{2}\right| \left| B^{n}\right|$,
    we establish a sharp Sobolev inequality for submanifolds of codimensional 2 as follows:
    \begin{corollary}
    	Let $M$ be a complete noncompact manifold of dimension $n+2$
    	with nonnegative sectional curvature,
    	Let $\Sigma$ be a compact $n$-dimensional submanifold of $M^{n+2}$ with smooth boundary $\partial\Sigma$ 
    	(possibly  $\partial\Sigma=\varnothing$).
    	If $A$ is a smooth symmetric uniformly positive definite $\left( 0,2\right)-$tensor field on $\Sigma$,
    	then
    	\begin{equation*}
    		\int_{\Sigma}\sqrt{\left| \mathrm{div}_{\Sigma}A\right|^{2}+\left| \left\langle A,\Rmnum{2}\right\rangle \right|^{2}}
    		+\int_{\partial\Sigma}\left| A\left( \nu\right) \right| 
    		\geq
    		n\left| B^{n}\right|^{\frac{1}{n}}\theta^{\frac{1}{n}}
    		\left( \int_{\Sigma}\left( \mathrm{det}A\right)^{\frac{1}{n-1}}\right) ^{\frac{n-1}{n}},
    	\end{equation*}
    	where $\nu$ is the unit outer normal vector field on $\partial\Sigma$ with respect to $\Sigma$, 
        $\theta$ is the asymptotic volume ratio of $M$.
    \end{corollary}

    According to S. Brendle's paper in $\cite{Bren}$,
    we know that the product $M^{n+1}\times\mathbb{R}$ and $M^{n+1}$ have the same asymptotic volume ratio,
    then we can view $\Sigma^{n}$ as a submanifold
    of the $(n+2)-$dimensional manifold $M\times\mathbb{R}$.  
    Hence the inequality in corollary 1.2 also holds in the codimension 1 setting: 

    \begin{corollary}	
    	Let $M$ be a complete noncompact manifold of dimension $n+1$
    	with nonnegative sectional curvature,
    	Let $\Sigma$ be a compact $n$-dimensional submanifold of $M^{n+1}$ with smooth boundary $\partial\Sigma$ 
    	(possibly  $\partial\Sigma=\varnothing$).
    	If $A$ is a smooth symmetric uniformly positive definite $\left( 0,2\right)-$tensor field on $\Sigma$,
    	then
    	\begin{equation*}
    		\int_{\Sigma}\sqrt{\left| \mathrm{div}_{\Sigma}A\right|^{2}+\left| \left\langle A,\Rmnum{2}\right\rangle \right|^{2}}
    		+\int_{\partial\Sigma}\left| A\left( \nu\right) \right| \geq
    		n\left| B^{n}\right|^{\frac{1}{n}}\theta^{\frac{1}{n}}
    		\left( \int_{\Sigma}\left( \mathrm{det}A\right)^{\frac{1}{n-1}}\right) ^{\frac{n-1}{n}},
    	\end{equation*}
    	where $\nu$ is the unit outer normal vector field on $\partial\Sigma$ with respect to $\Sigma$.	
    \end{corollary}

    As applications, we can first get the following S. Brendle's Michael-Simon-Sobolev inequality in $\cite{Bren}$:
    
    \begin{corollary}	\label{Re:02}
    (Brendle)	
    Let $M$ be a complete noncompact manifold of dimension $n+m$
    with nonnegative sectional curvature.
    Let $\Sigma$ be a compact submanifold of $M$ of dimensional $n$
    (possibly with boundary $\partial\Sigma$),
    and let $f$ be a positive smooth function on $\Sigma$. 
    If $m\geq2$, then
    \begin{equation}  \label{eq:1.2}
    	\int_{\Sigma}\sqrt{
    	\left| \nabla^{\Sigma}f\right|^{2}
    	+f^{2}\left| H\right|^{2}}
    	+\int_{\partial\Sigma}f 
    	\geq n\left[ \frac{\left( n+m\right)\left| B^{n+m}\right|}{m\left| B^{m}\right|}\right] ^{\frac{1}{n}}\theta^{\frac{1}{n}}
    	\left( \int_{\Sigma}f^{\frac{n}{n-1}}\right) ^{\frac{n-1}{n}},
    \end{equation}
    where $\theta$ denotes the asymptotic volume ratio of $M$.
    \end{corollary}

    Based on Corollary 1.4 above, we get the same isoperimetric inequality on minimal submanifolds (see \cite{Bren}).
    Moreover, we obtain another application concerning the nonexistence of the closed minimal submanifold,
    similar to the work of C. Y. Yi and Y. Zheng (see \cite{YI}).

    \begin{corollary}	\label{Re:03}
    (Yi-Zheng)	
    If $\left(M, g\right) $ is a complete noncompact Riemannian manifold of dimension $n+m \left(m\geq 2\right)$
   	with nonnegative sectional curvature and Euclidean volume growth $\left( \theta>0\right)$, 
   	then there doesn't exist any closed minimal submanifold in $M$.
    \end{corollary}

    This paper is organized as follows. 
    In section 2, 
    we introduce fundamental concepts and establish a generalized trace inequality for the product of two square matrices.
    In section 3, we provide the proof of Theorem 1.1.
    In section 4, we present the proofs of Corollary 1.4 and Corollary 1.5.
    
	\section{Preliminaries}
    Let $(M, \bar{g})$ be a complete noncompact $(n+m)-$dimension manifold with nonnegative Ricci curvature.
    Let $\Sigma$ be a compact $n$-dimensional submanifold of $M^{n+m}$ with smooth boundary $\partial\Sigma$ (possibly  $\partial\Sigma=\varnothing$).
    Denote by $\bar{D}$ the Levi-Civita connection on $(M, \bar{g})$,
    and by $\bar{R}$ the Riemann curvature tensor of $(M, \bar{g})$.
    We assume that $g_{\Sigma}$ is the induced Riemannian metric on $\Sigma$,
    $D_{\Sigma}$ is the Levi-Civita connection on $\Sigma$, 
    and $\nabla^{\Sigma}$ is the gradient of $\Sigma$.
    Let $A$ be a smooth symmetric uniformly positive definite $\left( 0,2\right)-$tensor field on $\Sigma$.
    For each point $x\in\Sigma$, we denote by $T_{x}\Sigma$ and $T_{x}^{\bot} \Sigma$ the tangent and normal space to $\Sigma$ at $x$, respectively.
    Let $\left( x^{1}, \dots, x^{n}\right) $ be a local coordinate system on $\Sigma$,
    the divergence of $A$ on $\Sigma$ is defined by
    \begin{equation*}
    	\mathrm{div}_{\Sigma}A:=
    	g_{\Sigma}^{ki}D^{\Sigma}_{k}A_{ij}dx^{j}.
    \end{equation*}
    Let $T$ and $S$ be two $\left(0,2\right)$-tensor fields on $\Sigma$. In general, the inner product of $T$ and $S$ can be written as
    \begin{equation*}
    	\left\langle T, S\right\rangle 
    	=g_{\Sigma}^{ik}g_{\Sigma}^{jl}T_{ij}S_{kl}
    	=T_{ij}S^{ij}.
    \end{equation*}
    The composition of $T$ and $S$ is the $\left(0,2\right)$-tensor $T\circ S$ defined by 
    \begin{equation*}
    (T\circ S)_{ij}=g^{kl}_{\Sigma}T_{ik}S_{lj}.
    \end{equation*}
    The $\mathrm{det}T$ signifies the determinant of $T$, 
    which is defined by the determinant of $\left(1,1\right)$-tensor $g_{\Sigma}^{ik}T_{jk}\frac{\partial}{\partial x^{i}}\otimes dx^{j}$. 
    When $T\circ S=g_{\Sigma}$, 
    we refer to $T$ as the inverse tensor of $S$ denoted by $T^{-1}$.
    Meanwhile, 
    $\Rmnum{2}$ denotes the second fundamental form of $\Sigma$ as defined by
     \begin{equation*}
     	\left\langle \Rmnum{2}\left( X,Y\right), Z\right\rangle :=
     	\left\langle \bar{D}_{X}Y, Z\right\rangle =
     	-\left\langle \bar{D}_{X}Z, Y\right\rangle,
     \end{equation*}
    where $X$ and $Y$ are tangent vector fields on $\Sigma$,
    $Z$ is a normal vector field to $\Sigma$. 
    Further, $\left\langle A, \Rmnum{2}\right\rangle\left( x\right) $ is the normal vector at $x\in\Sigma$ defined by
    \begin{equation*}
    	\left\langle A, \Rmnum{2}\right\rangle \left( x\right) =
         g_{\Sigma}^{ik}g_{\Sigma}^{jl}A_{ij}\Rmnum{2}(\frac{\partial}{\partial x ^{k}},\frac{\partial}{\partial x ^{l}}).
    \end{equation*}
    At last, we list the following lemma which extends the arithmetic-geometric mean inequality to the product of two square matrices version:
    \begin{lemma}(Lemma A.1 in \cite{DP})	
   	For $n\in\mathbb{N}$, let $A$ and $B$ be square symmetric matrices of size $n$.
   	Assume that $A$ is positive definite and $B$ is non-negative definite. Then
    	\begin{equation*}
    		\mathrm{det}AB
    		\leq\left(\frac{\mathrm{tr}\left( AB\right)}{n} \right)^{n}.
    	\end{equation*}
    	The equality holds if and only if $AB=\lambda I_{n}$ for some $\lambda\geq0$, 
    	where $I_{n}$ is the identity matrix.
    \end{lemma}

	\section{Proof of Theorem 1.1}

    First, we prove Theorem 1.1 in the special case that $\Sigma$ is connected.   
    Due to the scaling invariant property of $\eqref{eq:1.1}$,
    it suffices to establish $\eqref{eq:1.1}$ under the following condition:
    \begin{equation} \label{eq:2.1}
    	\int_{\Sigma}\sqrt{\left| \mathrm{div}_{\Sigma}A\right|^{2}+\left| \left\langle A,\Rmnum{2}\right\rangle \right|^{2}}
    	+\int_{\partial\Sigma}\left| A\left( \nu\right) \right| =
    	n \int_{\Sigma}\left( \mathrm{det}A\right) ^{\frac{1}{n-1}}.
    \end{equation}
    Substituting $\eqref{eq:2.1}$ into $\eqref{eq:1.1}$, we obtain
    \begin{equation*} 
    	n \int_{\Sigma}\left( \mathrm{det}A\right) ^{\frac{1}{n-1}}
    	\geq n\left[ \frac{\left( n+m\right)\left| B^{n+m}\right|}{m\left| B^{m}\right|}\right] ^{\frac{1}{n}}\theta^{\frac{1}{n}}\left( \int_{\Sigma}\left( \mathrm{det}A\right)^{\frac{1}{n-1}}\right) ^{\frac{n-1}{n}},
    \end{equation*}
     Divide both sides of the inequality by $\left( \int_{\Sigma}\left(  \mathrm{det}A\right)^{\frac{1}{n-1}}\right) ^{\frac{n-1}{n}}$ simultaneously, we have
     \begin{equation} \label{eq:2.2}
     	\left( \int_{\Sigma}\left( \mathrm{det}A\right)^{\frac{1}{n-1}}\right) ^{\frac{1}{n}}
     	\geq \left[ \frac{\left( n+m\right)\left| B^{n+m}\right|}{m\left| B^{m}\right|}\right] ^{\frac{1}{n}}\theta^{\frac{1}{n}}.
     \end{equation}
    In other words, we are left to prove $\eqref{eq:2.2}$ 
    instead of $\eqref{eq:1.1}$ in the following proof 
    under the condition $\eqref{eq:2.1}$.
  
    Given that $\Sigma$ is connected and $A$ satisfies $\eqref{eq:2.1}$, 
    we can find a classical solution $u:\Sigma\rightarrow\mathbb{R}$ to the following equation using the existence theorem of solutions for elliptic equations
    $\left( \mathrm{see} \cite{MT}\right)$
   	\begin{equation}  \label{eq:2.3}
    	\left\{\begin{aligned}
    		&\mathrm{div}_{\Sigma}\left( A\left( \nabla^{\Sigma}u\right) \right) \left( x\right) 
    		=n\left( \mathrm{det}A\left( x\right) \right) ^{\frac{1}{n-1}}
    		-\sqrt{\left| \mathrm{div}_{\Sigma}A\right|^{2}\left( x\right) + \left| \left\langle A,\Rmnum{2}\right\rangle \right|^{2}\left( x\right) },\ \mathrm{in}\  \Sigma\backslash\partial\Sigma,  \\
    		&\left\langle A\left( \nabla^{\Sigma}u\right) \left( x\right),\nu\left( x\right) \right\rangle =\left| A\left( \nu\left( x\right)\right) \right|,\ \mathrm{on}\  \partial\Sigma.
    	\end{aligned} \right.
    \end{equation}
     We define
    \begin{equation*}
    	\begin{split}
        &\Omega:=\left\lbrace x\in\Sigma\backslash\partial\Sigma:
        \left| \nabla^{\Sigma}u\left( x\right) \right|<1\right\rbrace,\\  
    	&U:=\{ \left( x,y\right): x\in\Sigma\backslash\partial\Sigma,
    	y\in T_{x}^{\bot}\Sigma,
    	\left| \nabla^{\Sigma}u\left( x\right) \right|^{2}+\left| y\right| ^{2}<1\}.\\  
        \end{split}
        \end{equation*} 
        In the following,
        we can fix a positive number $r$ and define a contact set
        \begin{equation*}
       	V_{r}:=\left\lbrace \left( \bar{x},\bar{y}\right)\in U: 
        ru\left( x\right) +\frac{1}{2}d\left( x,\mathrm{exp}_{\bar{x}}\left( rD_{\Sigma}u\left( \bar{x}\right)+r\bar{y}\right) \right) ^{2}
        \geq ru\left(\bar{x}\right) +\frac{1}{2}r^{2} \left( \left| D_{\Sigma}u\left( \bar{x}\right)\right| ^{2}+\left| \bar{y}\right| ^{2} \right)\right\rbrace
        \end{equation*} 
    for all  $x\in\Sigma$. 
    We denote a transport map $\Phi_{r}:U\rightarrow M$ by
    \begin{equation*}
    	\Phi_{r}\left( x, y\right) =\mathrm{exp}_{x}\left( r\nabla^{\Sigma}u\left( x\right)+ry\right) 
    \end{equation*} 
    for all $\left( x, y\right)\in U$.
    Standard elliptic regularity theory implies that the function $u\in C^{2, \gamma}(\Sigma)$ and $\Phi_{r}$ is of class  $C^{1, \gamma}$ 
    for each $0<\gamma<1 \left( \mathrm{see} \cite{DG}\right)$.

   \begin{lemma}
	The set
	\begin{equation*}
		\left\lbrace p\in M: d(x, p)<r \text{ for all } x\in\Sigma\right\rbrace
	\end{equation*} 
    is contained in $\Phi_{r}\left( V_{r}\right)$.
    \end{lemma}

    \begin{proof}
    The set is non-empty when r is large enough.
    Fix a point $p\in M$ with the property that $d(x, p)<r$ for all $x\in\Sigma$.	
    Since $\left\langle A\left( \nabla^{\Sigma}u\right),\nu \right\rangle 
    =\left\langle \nabla^{\Sigma}u, A\left(\nu\right) \right\rangle
    =\left| A\left( \nu\right) \right|$, 	
    it is easy to show that $\left\langle \nabla^{\Sigma}u,\frac{ A\left( \nu\right)}{\left| A\left( \nu\right)\right|}\right\rangle=1$.
    Let $\bar{\nu}=\frac{ A\left( \nu\right)}{\left| A\left( \nu\right)\right|}$,
    we get $\left\langle \nabla^{\Sigma}u,\bar{\nu}\right\rangle=1$.
    For every $x_{0}\in\partial\Sigma$,                                                     there exists a smooth curve $C:\tilde{\gamma}=\tilde{\gamma}\left( t\right)$, $t\in \left[0, \varepsilon \right)$ 
    satisfying
    $ \tilde{\gamma}\left( 0\right)=x_{0}$ when $t=0$,  
    $ \tilde{\gamma}^{'}\left( 0\right)=-\bar{\nu}\left( x_{0}\right)$.
    Since $\left\langle\nu, \bar{\nu} \right\rangle
    =\left\langle\nu, \frac{ A\left( \nu\right)}{\left| A\left( \nu\right)\right|}\right\rangle >0$,
    we can find that $-\nu$ points to the interior of $\partial\Sigma$,
    this implies that the curve $C$ moves toward the interior of $\Sigma$ from $x_{0}\in\partial\Sigma$.
    Let $f\left( x\right)=ru\left( x\right)+\frac{1}{2}d^{2}\left(x, p \right)$. 
    Additionally, we need the following proposition:
    
    \begin{proposition}
    	The function $f\left( x\right)$ attains its minimum in the interior of $\Sigma$.
    \end{proposition}
    \begin{proof}
    If $x_{0}\in\partial\Sigma$ is a smooth point of the distance function $d\left( x, p\right)$.
    Along the curve $C$, we observe that:
   	\begin{equation*}
    	\begin{split}
    		\dfrac{d}{dt}\left[ ru\left( x\right)+\frac{1}{2}d^{2}\left(x, p \right)\right] \bigg|_{t=0}
    		&=\left[r \left\langle \nabla^{\Sigma}u, \tilde{\gamma}^{'}\left( t\right) \right\rangle +d\left(x, p \right)\left\langle \nabla d, \tilde{\gamma}^{'}\left(  t\right) \right\rangle \right] \bigg|_{t=0} \\
    		&=r \left\langle \nabla^{\Sigma}u, -\bar{\nu} \right\rangle +d\left(x, p \right)\left\langle \nabla d,-\bar{\nu} \right\rangle \\
    		&=-r-d\left(x, p \right)\left\langle \nabla d, \bar{\nu} \right\rangle \\
    		&\leq-r+d\left(x, p \right)\\
    		&<0.
    	\end{split}
    \end{equation*}
    The penultimate inequality holds due to the fact that $-\left\langle \nabla d, \bar{\nu} \right\rangle\leq\left| \nabla d\right| \left|\bar{\nu}  \right|$. 
    Consequently, $f^{'}\left( x\right)<0$, 
    implying that $f\left( x\right)$ attains its minimum in the interior of $\Sigma$.
    If $x_{0}\in\partial\Sigma$ is a cut point of $p$,
    there exists a minimal geodesic $\bar{\gamma}$ 
    connecting the two point $x_{0}$ and $p$,  
    advanced the point $p$ to a new point $p_{0}$ along the geodesic $\bar{\gamma}$, 
    it follows that the point $p_{0}$ to the point $x_{0}$ must be the unique minimal geodesic.
    We conclude that
    \begin{equation*}
    	\begin{split}
    		\dfrac{d}{dt}\left[ ru\left( x\right)+\frac{1}{2}d^{2}\left(x, p \right)\right] \bigg|_{t=0}
    		&=\dfrac{d}{dt}\left[ ru\left( x\right)+\frac{1}{2}\left( d\left(x, p_{0} \right)+ d\left( p_{0}, p\right)\right) ^{2} \right] \bigg|_{t=0}\\
    		&=\left[r \left\langle \nabla^{\Sigma}u, \bar{\gamma}^{'}\left( t\right) \right\rangle +
    		\left( d\left(x, p_{0} \right)+ d\left( p_{0}, p\right)\right) 
    	    \nabla d\left( x, p_{0}\right) \right] \bigg|_{t=0} \\
    		&=r\left\langle \nabla^{\Sigma}u, -\bar{\nu} \right\rangle 
    		+d\left(x, p \right)\left\langle \nabla d,-\bar{\nu} \right\rangle \\
    		&=-r-d\left(x, p \right)\left\langle \nabla d, \bar{\nu} \right\rangle \\
    		&\leq-r+d\left(x, p \right)\\
    		&<0.
    	\end{split}
    \end{equation*}
    Thus $f^{'}\left( x\right)<0$, indicating that
    minimal point of $f\left( x\right)$ lies in $\Sigma\backslash \partial\Sigma$.
    \end{proof}
    Based on the above Proposition 3.2,
    we can choose $\bar{x}\in \Sigma\backslash \partial\Sigma$ as the point where the function $f\left( x\right)=ru\left( x\right)+\frac{1}{2}d^{2}\left(x, p \right)$ reaches its minimum.
    Let $\bar{\gamma}:\left[ 0, 1\right]\rightarrow M$ be a minimizing geodesic such that $\bar{\gamma}\left(  0\right)=x_{0}$ and $\bar{\gamma}\left( 1\right)=p$.
    For every path $\gamma:\left[ 0, 1\right]\rightarrow M$ satisfying 
    $\gamma\left( 0\right)\in \Sigma$ and
    $\gamma\left( 1\right)=p$,
    we obtain
    \begin{equation*}
    	\begin{split}
    		ru\left( \gamma\left( 0\right)\right)
    		+\frac{1}{2}\int_{0}^{1}\left| \gamma^{'}\left( t\right) \right|^{2}dt
    		&\geq ru\left( \gamma\left( 0\right)\right)
    		+\frac{1}{2}d\left( \gamma\left( 0\right), p\right)^{2} \\
    		&\geq ru\left( \bar{x}\right)
    		+\frac{1}{2}d\left( \bar{x}, p\right)^{2} \\
    		&=ru\left( \bar{\gamma}\left( 0\right)\right)
    		+\frac{1}{2}\left| \bar{\gamma}^{'}\left( 0\right) \right|^{2}\\
    		&=ru\left(\bar{\gamma}\left( 0\right)\right)
    		+\frac{1}{2}\int_{0}^{1}\left|\bar{\gamma}^{'}\left( t\right) \right|^{2}dt. \\
    	\end{split}
    \end{equation*}
    In other words, the path $\bar{\gamma}$
    minimizes the functional $ru\left( \gamma\left( 0\right)\right)
    +\frac{1}{2}\int_{0}^{1}\left| \gamma^{'}\left( t\right) \right|^{2}dt$ among all paths $\gamma:\left[ 0, 1\right]\rightarrow M$ satisfying 
    $\gamma\left( 0\right)\in \Sigma$ and
    $\gamma\left( 1\right)=p$.
    Hence, the formula for the first variation of energy
    implies
	\begin{equation*}
	  \bar{\gamma}\left( 0\right)+r\nabla^{\Sigma}u\left(\bar{x}\right)
	  \in T_{x}^{\bot}\Sigma.
	\end{equation*}
    Consequently, we can find a vector $\bar{y}\in T_{x}^{\bot}\Sigma$ 
    such that
	\begin{equation*}
		\bar{\gamma}\left( 0\right)=r\nabla^{\Sigma}u\left(\bar{x}\right)
		+r\bar{y}.
	\end{equation*} 
	This implies
	\begin{equation*}
		r^{2}\left( \left| \nabla^{\Sigma}u\left(\bar{x}\right)\right| ^{2}+\left| \bar{y}\right| ^{2}\right) 
		=\left|\bar{\gamma}\left( 0\right) \right|^{2}
		=d\left( \bar{x}, p\right)^{2}
		<r^{2}.
	\end{equation*}
    Therefore, $\left| \nabla^{\Sigma}u\left(\bar{x}\right)\right| ^{2}+\left| \bar{y}\right| ^{2}<1$. 
    In other words, $\left(\bar{x}, \bar{y}\right)\in U$.
    Moreover,
    \begin{equation*}
    	\Phi_{r}\left( \bar{x}, \bar{y}\right)=\mathrm{exp}_{\bar{x}}\left( rD_{\Sigma}u\left( \bar{x}\right)+r\bar{y}\right)
    	=\mathrm{exp}_{\bar{\gamma}\left( 0\right)}\bar{\gamma}\left( 0\right)
    	=\bar{\gamma}\left( 1\right)=p.
    \end{equation*}
    At last, for each point $x\in \Sigma$, we deduce
    \begin{equation*}
    	\begin{split}
    	ru\left( x\right) +\frac{1}{2}d\left( x,\mathrm{exp}_{\bar{x}}\left( rD_{\Sigma}u\left( \bar{x}\right)+r\bar{y}\right) \right) ^{2}
    	&=ru\left( x\right)+\frac{1}{2}d\left( x, p\right)^{2}\\
    	&\geq ru\left(\bar{x}\right)+\frac{1}{2}d\left(\bar{x}, p\right)^{2}\\
    	&=ru\left( \bar{\gamma}\left( 0\right)\right)
    	+\frac{1}{2}\left| \bar{\gamma}^{'}\left( 0\right) \right|^{2}\\
 	    &\geq ru\left(\bar{x}\right) +\frac{1}{2}r^{2} \left( \left| D^{\Sigma}u\left( \bar{x}\right)\right| ^{2}+\left| \bar{y}\right| ^{2} \right).
 	    \end{split}
    \end{equation*}  
	Thus $\left(\bar{x}, \bar{y}\right)\in\Phi_{r}\left( V_{r}\right)$.
	This completes the proof of Lemma 3.1.
    \end{proof} 

    Using Lemma 3.1, similar to Brendle's paper,
    we can get the following lemma replacing the ball with the annulus.

    \begin{lemma}
    (\cite{Bren}, Lemma 3.5)	
    For each $0<\sigma<1$, the set
    \begin{equation*}
    	\left\lbrace p\in M: \sigma r<d(x, p)<r \text{ for all } x\in\Sigma\right\rbrace
    \end{equation*}	
    is contained in the set
    $\Phi_{r}\left( \left\lbrace 
    \left(x, y\right)\in V_{r}: \left|\nabla^{\Sigma}u \left( x\right) \right|^{2}+\left| y\right|^{2}>\sigma^{2}\right\rbrace\right)$.
    \end{lemma}

    \begin{proof}
   	Fix a point $ p\in M$ with the property that $\sigma r<d(x, p)<r$ for
   	all $x\in\Sigma$.
   	By Lemma 3.1, 
   	we can find a point $ \left(\bar{x}, \bar{y}\right)\in V_{r}$ such that
    \begin{equation*}
		\Phi_{r}\left( \bar{x}, \bar{y}\right)=\mathrm{exp}_{\bar{x}}\left( rD_{\Sigma}u\left( \bar{x}\right)+r\bar{y}\right)
		=p.
	\end{equation*}
	 This implies
	     \begin{equation*}
	 	\begin{split}
	 		\sigma^{2}r^{2}
	 		&<d\left(\bar{x}, p \right)^{2}\\
	 		&=d\left(\bar{x}, \mathrm{exp}_{\bar{x}}\left( rD_{\Sigma}u\left( \bar{x}\right)+r\bar{y}\right) \right)^{2}\\
	 		&\leq \left| rD^{\Sigma}u\left( \bar{x}\right)
	 		+r \bar{y}\right|^{2}\\
	 		&=r^{2}\left( \left| \nabla^{\Sigma}u \left( \bar{x}\right) \right|^{2}+\left|\bar{y} \right|^{2}\right). 
	 	\end{split}
	\end{equation*}  
    Therefore,
	$\left|\nabla^{\Sigma}u \left( x\right) \right|^{2}+\left| y\right|^{2}>\sigma^{2}$.
	This completes the proof of Lemma 3.3.
    \end{proof}

    To complete the proof of Theorem 1.1, we also need the following three lemmas in \cite{Bren}. 
    
    \begin{lemma}
    	(\cite{Bren}, Lemma 3.1)	
     Suppose that $\left( \bar{x},\bar{y}\right)\in V_{r}$ and
     let $\bar{\gamma}\left( t\right):=\mathrm{exp}_{\bar{x}}
     \left( rt\nabla^{\Sigma}u\left( \bar{x}\right)+rt\bar{y}\right)$ 
     for $t\in \left[ 0, 1\right]$. If $Z$ is a vector field along $\bar{\gamma}$ satisfying $Z\left( 0\right)\in T_{x}^{\bot}\Sigma$ and $Z\left( 1\right) =0$, then
    	\begin{equation*}
    	\begin{split}	
         r\left( \nabla^{2}_{\Sigma}u\right) \left(Z\left( 0\right), Z\left( 0\right) \right)-r\left\langle\Rmnum{2}\left(Z\left( 0\right), Z\left( 0\right)\right), \bar{y} \right\rangle \\ +\int_{0}^{1}\left| \bar{D}_{t}Z\left( t\right) \right| ^{2}
         -\bar{R}\left(\bar{\gamma}\left( t\right), Z\left( t\right) , \bar{\gamma}\left( t\right), Z\left( t\right) \right) \geq0.
        \end{split}	
    	\end{equation*}		
    \end{lemma}

     \begin{lemma}
    	(\cite{Bren}, Lemma 3.2)	
      Suppose that $\left( \bar{x},\bar{y}\right)\in V_{r}$. Then 
      $g+r\nabla^{2}_{\Sigma}u\left(\bar{x}\right)-r\left\langle\Rmnum{2}\left( \bar{x}\right),\bar{y}\right\rangle \geq0$.
     \end{lemma}
 
     \begin{lemma}
  	  (\cite{Bren}, Lemma 3.3)	
    Suppose that $\left( \bar{x},\bar{y}\right)\in V_{r}$ and
  	let $\bar{\gamma}\left( t\right):=\mathrm{exp}_{\bar{x}}
  	\left( rt\nabla^{\Sigma}u\left( \bar{x}\right)+rt\bar{y}\right)$ 
  	for $t\in \left[ 0, 1\right]$.  
  	Moreover, let $\left\lbrace e_{1}, \ldots, e_{n}\right\rbrace$ be an orthonormal basis of $T_{x}^{\bot}\Sigma$.
  	Suppose that $W$ is a Jacobi field along $\bar{\gamma}$ satisfying  $W\left( 0\right)\in T_{x}^{\bot}\Sigma$ and
  	$\left\langle \bar{D}_{t}W\left( 0\right), e_{j}\right\rangle=
  	r\left( \nabla^{2}_{\Sigma}u\right) \left(W\left( 0\right), e_{j}\right)-r\left\langle\Rmnum{2}\left(W\left( 0\right), e_{j}\right),\bar{y}\right\rangle$
  	for each $1\leq j\leq n$.
  	If $W\left( \tau\right)=0$ for some $0<\tau<1$, 
  	then $W$ vanishes identically.
     \end{lemma}
 
    Among these, Lemma 3.5 holds when $M$ has a nonnegative sectional curvature.
    Based on the above three lemmas, we also obtain the following lemma.
    
    \begin{lemma}	\label{Le:2.20}
    (\cite{Bren}, Lemma 3.6)
     The Jacobian determinant of $\Phi_{r}$ is given by
      \begin{equation*}
      	\left| \mathrm{det}D\Phi_{r}\left( x, y\right)\right|  
      	\leq r^{m}\mathrm{det}\left(g+rD_{\Sigma}^{2}u\left( x\right)
      	-r\left\langle \Rmnum{2}\left( x\right), y\right\rangle \right) 
      \end{equation*} 
    for all $\left( x,y\right) \in V_{r}$.
    \end{lemma}

     \begin{lemma}	\label{Le:2.30}
    	The Jacobian determinant of $\Phi_{r}$ satisfies
    	\begin{equation*}
    	   \left| \mathrm{det}D\Phi_{r}\left( x, y\right)\right|  
    		\leq r^{m}\dfrac{1}{\mathrm{det}A\left( x\right)}
    		\left(\dfrac{\mathrm{tr}_{g}A\left( x\right)}{n}+r\left( \mathrm{det}A\left( x\right) \right) ^{\frac{1}{n-1}}\right)^{n}
    	\end{equation*} 
    	for all $\left( x, y\right) \in V_{r}$.
    \end{lemma} 

    \begin{proof}
    Given a pair of points $\left( x, y\right) \in V_{r}$,
    by Cauchy-Schwarz inequality and 
    $\left| \nabla^{\Sigma}u\left( x\right) \right|^{2}+\left| y\right| ^{2}<1$,
    we obtain
    \begin{equation*}
    	\begin{split}
    		&-\left\langle \mathrm{div}_{\Sigma}A\left( x\right), \nabla^{\Sigma}u\left( x\right)\right\rangle
    		- \left\langle A\left( x\right) , \left\langle \Rmnum{2}\left( x\right), y\right\rangle\right\rangle  \\
    		=&-\left\langle \mathrm{div}_{\Sigma}A\left( x\right), \nabla^{\Sigma}u\left( x\right)\right\rangle
    		- \left\langle \left\langle A\left( x\right) , \Rmnum{2}\left( x\right)\right\rangle , y\right\rangle\\
    		=&-\left\langle \mathrm{div}_{\Sigma}A\left( x\right)+\left\langle A\left( x\right) , \Rmnum{2}\left( x\right)\right\rangle, \nabla^{\Sigma}u\left( x\right)+y\right\rangle\\
    		\leq&\sqrt{\left| \nabla^{\Sigma}u\left( x\right)\right| ^{2}+\left| y\right| ^{2}}
    		\sqrt{\left| \mathrm{div}_{\Sigma}A\right| ^{2}\left( x\right)+\left| \left\langle A, \Rmnum{2}\right\rangle\right| ^{2}\left( x\right) }\\
    		\leq&\sqrt{\left| \mathrm{div}_{\Sigma}A\right| ^{2}\left( x\right)+\left| \left\langle A, \Rmnum{2}\right\rangle\right| ^{2}\left( x\right) }.
    	\end{split}
    \end{equation*}
    Note that 	
    \begin{equation*}
        \mathrm{div}_{\Sigma}\left( A\left( \nabla^{\Sigma}u\right) \right)
        =\left\langle \mathrm{div}_{\Sigma}A, \nabla^{\Sigma}u\right\rangle 
        + \left\langle A, D_{\Sigma}^{2}u\right\rangle.
    \end{equation*}
    According to the equation of $u$, we have
    \begin{equation*}
    	\begin{split}
    		&\left\langle A\left( x\right), D_{\Sigma}^{2}u\left( x\right)-\left\langle \Rmnum{2}\left( x\right), y\right\rangle\right\rangle \\
    		=&\mathrm{div}_{\Sigma}\left( A\left( \nabla^{\Sigma}u\right) \right)\left( x\right) 
    		-\left\langle \mathrm{div}_{\Sigma}A\left( x\right), \nabla^{\Sigma}u\left( x\right)\right\rangle
    		-\left\langle A\left( x\right) , \left\langle \Rmnum{2}\left( x\right), y\right\rangle\right\rangle \\
    		=&n\left( \mathrm{det}A\left( x\right) \right) ^{\frac{1}{n-1}}
    		-\sqrt{\left| \mathrm{div}_{\Sigma}A\right| ^{2}\left( x\right)+\left| \left\langle A, \Rmnum{2}\right\rangle\right| ^{2}\left(x\right)}	\\
    		&-\left\langle \mathrm{div}_{\Sigma}A\left( x\right), \nabla^{\Sigma}u\left( x\right)\right\rangle
    		- \left\langle A\left( x\right), \left\langle \Rmnum{2}\left( x\right), y\right\rangle\right\rangle \\
    		\leq& n\left( \mathrm{det}A\left( x\right) \right) ^{\frac{1}{n-1}}.
    	\end{split}
    \end{equation*}
    Since $A$ is symmetric positive definite, we derive by Lemma 3.7 and Lemma 2.1 that
    \begin{equation*}
    	\begin{split}
    		&\left| \mathrm{det}D\Phi_{r}\left( x, y\right)\right| \\
    		\leq& r^{m}\mathrm{det}\left(g+rD_{\Sigma}^{2}u\left( x\right)
    		-r\left\langle \Rmnum{2}\left( x\right), y\right\rangle \right)\\
    		=& r^{m}\dfrac{1}{\mathrm{det}A\left( x\right)}\mathrm{det}
    		\left[A\left( x\right) \circ \left(g+rD_{\Sigma}^{2}u\left( x\right)-r\left\langle \Rmnum{2}\left( x\right), y\right\rangle \right)\right] \\ 
    		\leq& r^{m}\dfrac{1}{\mathrm{det}A\left( x\right)}
    		\left(\dfrac{\mathrm{tr}\left( A\left( x\right) \left(g+rD_{\Sigma}^{2}u\left( x\right)-r\left\langle \Rmnum{2}\left( x\right), y\right\rangle \right)\right) }{n}\right)^{n}\\
    		=& r^{m}\dfrac{1}{\mathrm{det}A\left( x\right)}
    		\left(\dfrac{\left\langle A\left( x\right), g\right\rangle }{n}+\dfrac{ r\left\langle A\left( x\right), D_{\Sigma}^{2}u\left( x\right)-\left\langle \Rmnum{2}\left( x\right), y\right\rangle \right\rangle}{n}\right)^{n}\\
    		\leq& r^{m}\dfrac{1}{\mathrm{det}A\left( x\right)}
    		\left(\dfrac{\mathrm{tr}_{g}A\left( x\right)}{n}+r\left( \mathrm{det}A\left( x\right) \right) ^{\frac{1}{n-1}}\right)^{n}\\
    	\end{split}
    \end{equation*} 

    This lemma follows.
    
    \end{proof}
    \textbf{\emph{Proof of Theorem 1.1.}}
    Given a constant $\sigma$ such that $0<\sigma<1$, 
    by Lemma 3.2 and Lemma 3.8, we conclude that
    \begin{equation*}
    	\begin{split}
    		&\left\lbrace p\in M: \sigma r<d(x, p)<r \text{ for all } x\in\Sigma\right\rbrace   \\
    		\leq& \int_{\Omega}\left( 
    		\int_{\left\lbrace y\in T_{x}^{\bot}\Sigma:\sigma^{2}<\left| \Phi_{r}\left( x, y\right) \right|^{2}+\left| y\right|^{2}<1\right\rbrace } 
    		\left| \mathrm{det}D\Phi_{r}\left( x, y\right) \right| 1_{V_{r}}\left( x, y\right)
    		dy\right) d\mathrm{vol}\left( x\right) \\
    		\leq& \int_{\Omega}\left( 
    		\int_{\left\lbrace y\in T_{x}^{\bot}\Sigma:\sigma^{2}<\left| \nabla^{\Sigma}u\left( x\right) \right|^{2}+\left| y\right| ^{2}<1\right\rbrace } 
    	     r^{m}\dfrac{1}{\mathrm{det}A\left( x\right)}
    	    \left(\dfrac{\mathrm{tr}_{g}A\left( x\right)}{n}+r\left( \mathrm{det}A\left( x\right) \right) ^{\frac{1}{n-1}}\right)^{n} dy\right)d\mathrm{vol}\left( x\right) \\
    		=& \left| B^{m}\right|\int_{\Omega} 
    		\left[\left( 1-\left| \nabla^{\Sigma}u\left( x\right) \right|^{2}\right)^{\frac{m}{2}}-\left( \sigma^{2}-\left| \nabla^{\Sigma}u\left( x\right) \right|^{2}\right)^{\frac{m}{2}}_{+}\right]\\ 
    		 &\qquad\qquad\qquad\qquad\qquad r^{m}\dfrac{1}{\mathrm{det}A\left( x\right)}
    		\left(\dfrac{\mathrm{tr}_{g}A\left( x\right)}{n}+r\left( \mathrm{det}A\left( x\right) \right) ^{\frac{1}{n-1}}\right)^{n}d\mathrm{vol}\left( x\right)\\
    		\leq&\frac{m}{2}\left| B^{m}\right|\left( 1-\sigma^{2}\right) 
    		\int_{\Omega}  r^{m}\dfrac{1}{\mathrm{det}A\left( x\right)}
    		\left(\dfrac{\mathrm{tr}_{g}A\left( x\right)}{n}+r\left( \mathrm{det}A\left( x\right) \right) ^{\frac{1}{n-1}}\right)^{n}d\mathrm{vol}\left( x\right). \\
    	\end{split}
    \end{equation*}
    for all $r>0$. Because of $m\geq2$, the last inequality uses the mean value theorem
    \begin{equation*}
    	b^{\frac{m}{2}}-a^{\frac{m}{2}}\leq\frac{m}{2}\left( b-a\right)  
    \end{equation*}
    for $0\leq a\leq b\leq1$.
    Dividing by $r^{n+m}$ and letting $r\rightarrow\infty$, we have
    \begin{equation}
    	\left| B^{n+m}\right|\left( 1-\sigma^{n+m}\right)\theta
    	\leq\frac{m}{2}\left| B^{m}\right|\left( 1-\sigma^{2}\right) 
    	\int_{\Omega} \left( \mathrm{det}A\right) 
    	^{\frac{1}{n-1}}.
    \end{equation}
    Next, we dividing both side by $1-\sigma$ and letting $\sigma\rightarrow1$, then
    \begin{equation}
    	\left( n+m\right)\left| B^{n+m}\right|\theta
    	\leq m\left| B^{m}\right| 
    	\int_{\Omega} \left( \mathrm{det}A\right) 
    	^{\frac{1}{n-1}}
    	\leq m\left| B^{m}\right| 
    	\int_{\Sigma} \left( \mathrm{det}A\right) 
    	^{\frac{1}{n-1}}.
    \end{equation}
    That is
    \begin{equation} 
        \int_{\Sigma}\left( \mathrm{det}A\right)^{\frac{1}{n-1}}
        \geq \left[ \frac{\left( n+m\right)\left| B^{n+m}\right|}{m\left| B^{m}\right|}\right]\theta.
    \end{equation}
    This means that the above inequality is $\eqref{eq:2.2}$.
    From this, according to the beginning of the analysis, we can get $\eqref{eq:1.1}$.
    Thus, in the special case that $\Sigma$ is connected, Theorem 1.1 has been proved.
    
    It remains to consider the case when $\Sigma$ is disconnected. 
    In this case, we apply the inequality to each individual connected component of $\Sigma$, and compute the sum across all connected components. Notic that 
    \begin{equation*}
    	a^{\frac{n-1}{n}}+b^{\frac{n-1}{n}}
    	>a\left( a+b\right) ^{\frac{1}{n}}+b\left( a+b\right) ^{\frac{1}{n}}
    	=\left( a+b\right) ^{\frac{n-1}{n}}
    \end{equation*}   
    for $a, b>0$.
    we can conclude that
    \begin{equation*}
		\int_{\Sigma}\sqrt{\left| \mathrm{div}_{\Sigma}A\right| ^{2} +\left| \left\langle A, \Rmnum{2}\right\rangle\right| ^{2}}
		+\int_{\partial\Sigma}\left| A\left( \nu\right) \right| 
		> n\left[ \frac{\left( n+m\right)\left| B^{n+m}\right|}{m\left| B^{m}\right|}\right] 
		^{\frac{1}{n}}\theta^{\frac{1}{n}}
		\left( \int_{\Sigma} \left( \mathrm{det}A\right) 
	  	^{\frac{1}{n-1}}\right) ^{\frac{n-1}{n}}
    \end{equation*}    
    if $\Sigma$ is disconnected.
    This completes the proof of Theorem 1.1.\hfill$\Box$\\

	\section{Proof of Corollary 1.4 And Corollary 1.5}
	
	First, let's begin the proof of Corollary 1.4.
	
	\textbf{\emph{Proof of Corollary 1.4.}} 
	According to Theorem 1.1, the following inequality holds:
	\begin{equation*} 
		\int_{\Sigma}\sqrt{\left| \mathrm{div}_{\Sigma}A\right|^{2}+\left| \left\langle A,\Rmnum{2}\right\rangle \right|^{2}}
		+\int_{\partial\Sigma}\left| A\left( \nu\right) \right| 
		\geq n\left[ \frac{\left( n+m\right)\left| B^{n+m}\right|}{m\left| B^{m}\right|}\right] ^{\frac{1}{n}}\theta^{\frac{1}{n}}\left( \int_{\Sigma}\left( \mathrm{det}A\right)^{\frac{1}{n-1}}\right) ^{\frac{n-1}{n}}.
	\end{equation*}
	When $A=fg_{\Sigma}$ for a positive smooth function $f$ on $\Sigma$, 
    we find that
    \begin{equation*}
        \mathrm{div}_{\Sigma}A
        =\mathrm{div}\left( fg_{\Sigma}\right) 
        =f\mathrm{div}g_{\Sigma}+\left\langle \nabla^{\Sigma}f, g\right\rangle 
        =\nabla^{\Sigma}f.
    \end{equation*}
    Let $H$ denote the mean curvature vector of $\Sigma$, we have
    \begin{equation*}
    	\left\langle A,\Rmnum{2}\right\rangle
    	=\left\langle fg_{\Sigma}, \Rmnum{2}\right\rangle
    	=fH.
    \end{equation*}
    It's readily apparent that $A\left( \nu\right) :=A\left( \cdot, \nu\right)=fg\left( \nu\right)$. 
    Let $X$ be a tangent vector field on $M$, then
    $g\left( \nu\right)\left( X\right)=g\left( \nu, X\right)$,
    since $\left| \nu\right| =1$, 
    it is easy to show that
    $\left| A\left( \nu\right)\right| 
    =f\left| g\left( \nu\right)\right| =f$.
    In order to calculate the determinant of $A$,
    we first transform $A$ from a (0, 2)-tensor field to a (1, 1)-tensor field, 
    following this transformation, 
    we proceed to compute the determinant of $A$. 
    i. e. $A^{k}_{i}:=g^{jk}A_{ij}$, 
    $\tilde{g}=g^{jk}g_{ij}=\delta^{k}_{i}$.
    Consequently, we obtain 
    \begin{equation*}
        \mathrm{det}A
        =\mathrm{det}\left( fg_{\Sigma}\right) 
        =f^{n}\mathrm{det}g_{\Sigma}
        =f^{n}\mathrm{det}\tilde{g}
        =f^{n}.
    \end{equation*}
	Therefore, we arrive at the inequality as outlined in Theorem 1.3 of S. Brendle’s seminal paper in \cite{Bren}:
	\begin{equation*} 
		\int_{\Sigma}\sqrt{
		\left| \nabla^{\Sigma}f\right|^{2}
		+f^{2}\left| H\right|^{2}}
		+\int_{\partial\Sigma}f 
		\geq n\left[ \frac{\left( n+m\right)\left| B^{n+m}\right|}{m\left| B^{m}\right|}\right] ^{\frac{1}{n}}\theta^{\frac{1}{n}}
		\left( \int_{\Sigma}f^{\frac{n}{n-1}}\right) ^{\frac{n-1}{n}}.
	\end{equation*}
    \hfill$\Box$\\

	At last, we have the following proof of Corollary 1.5.
	
	\textbf{\emph{Proof of Corollary 1.5.}}	
	The contradiction can be given. 
	Assume that $(\Sigma, g_{\Sigma})$ is a closed n-dimension minimal 
	submanifold of $M$. 
	For Theorem 1.1, we can set $A=g$,
	and thus conclude that	
	\begin{equation*} 
		0\geq n\left[ \frac{\left( n+m\right)\left| B^{n+m}\right|}{m\left| B^{m}\right|}\right] ^{\frac{1}{n}}\theta^{\frac{1}{n}}
		\left| \Sigma\right|^{\frac{n-1}{n}}>0.
	\end{equation*}	
	It's a contradiction.
	This contradiction completes the proof.
	\hfill$\Box$\\

	   {\bf Acknowledgement.}

	This work was supported in part by 
	the National Natural Science Foundation of China (No. 12271163)
	and Science and Technology Commission of Shanghai Municipality (No. 22DZ2229014). 
	Additional support was provided by the School of Mathematical Sciences, Key Laboratory of MEA (Ministry of Education) and Shanghai Key Laboratory PMMP (No. 18DZ2271000), East China Normal University, Shanghai 200241, China.

\end{document}